\documentclass[a4paper,12pt]{article}
    \usepackage{amsmath}
    \usepackage{amssymb}
    \usepackage{amsthm}

\usepackage{graphicx,psfrag}
\usepackage[pdf]{pstricks}
\usepackage{subfigure}

\usepackage{anysize}
 \marginsize{1.4cm}{1.4cm}{1.4cm}{1.4cm}

      \newcommand {\al}   {\alpha}          \newcommand {\bt}  {\beta}
      \newcommand {\gam } {\gamma}          
      \newcommand {\del}  {\delta}          \newcommand {\Del} {\Delta}
              \newcommand {\ve}   {\varepsilon}

      \newcommand {\om}   {\omega}          \newcommand {\Om}  {\Omega}
      \newcommand {\pl}   {\partial}        \newcommand {\s}    {\sigma}
           \newcommand {\UUU}  {{\cal U}}
      \newcommand {\RRR}  {{\mathbb R}}     
      \newcommand {\OOO}  {{\cal O}}        
      \newcommand {\ZZZ}  {{\mathbb Z}}     
              
             \newcommand {\kap}  {\kappa}
      \newcommand {\FFF}  {{\cal F}}

      \newcommand {\bbss}  {\begin{slide}}
      \newcommand {\eess}  {\end{slide}}

             \newcommand {\bbb}{s}

\DeclareMathOperator{\graph}{graph}

      \newtheorem{theorem}{Theorem}
      \newtheorem{lemma}{Lemma}
     \newtheorem{proposition}{Proposition}
      \newtheorem{corollary}{Corollary}
      
       \newtheorem{remark}{Remark}

\newcounter{fig}
\renewcommand{\figure}{\refstepcounter{fig}%
                  Fig. \arabic{fig}. }

\title{Minimal resistance of curves under the single impact assumption}

\author{Arseniy Akopyan\thanks{Dept. of Mathematics, Moscow Institute of Physics and Technology, Institutskiy per. 9, Dolgoprudny, Russia 141700 and Institute for Information Transmission Problems RAS, Bolshoy Karetny per. 19, Moscow, Russia 127994} \and Alexander Plakhov\thanks{Center for R\&{}D in Mathematics and Applications, Department of Mathematics, University of Aveiro, Portugal and Institute for Information Transmission Problems, Moscow, Russia}}

\begin{document}
\maketitle

\begin{abstract}
We consider the hollow on the half-plane $\{ (x,y) : y \le 0 \} \subset \RRR^2$ defined by a function $u : (-1,\, 1) \to \RRR$,\, $u(x) < 0$ and a vertical flow of point particles incident on the hollow. It is assumed that $u$ satisfies the so-called single impact condition (SIC): each incident particle is elastically reflected by $\graph(u)$ and goes away without hitting the graph of $u$ anymore. We solve the problem: find the function $u$ minimizing the force of resistance created by the flow. We show that the graph of the minimizer is formed by two arcs of parabolas symmetric to each other with respect to the $y$-axis. Assuming that the resistance of $u \equiv 0$ equals $1$, we show that the minimal resistance equals $\pi/2 - 2\arctan(1/2) \approx 0.6435$. This result completes the previously obtained result \cite{PkakSIAM} stating in particular that the minimal resistance of a hollow in higher dimensions equals $0.5$.

{ We additionally consider a similar problem of minimal resistance, where the hollow in the half-space $\{(x_1,\ldots,x_d, y) : y \le 0 \} \subset \RRR^{d+1}$ is defined by a radial function $U$ satisfying SIC, $U(x) = u(|x|)$, with $x = (x_1,\ldots,x_d)$,\, $u(\xi) < 0$ for $0 \le \xi < 1$ and $u(\xi) = 0$ for $\xi \ge 1$, and the flow is parallel to the $y$-axis. The minimal resistance is greater than $0.5$ (and coincides with $0.6435$ when $d = 1$) and converges to $0.5$ as $d \to \infty$.}
\end{abstract}

\begin{quote}
\end{quote}

\begin{quote}
\end{quote}

\section{Introduction}

Consider a function $u : \bar\Om \to \RRR$, where $\Om \subset \RRR^d$,\ $d \ge 1$ is an open connected bounded set. We assume that the gradient $\nabla u(x)$  exists and is continuous on an open full-measure subset of $\Om$, and
\begin{equation}\label{eq1}
u(x) = 0 \ \ \text{for}\ \, x \in \pl\Om \quad \text{and} \quad u (x) \le 0 \ \ \text{for}\ \, x \in \Om.
\end{equation}
Consider a parallel flow of point particles in $\RRR^{d+1}$ incident on the graph of $u$ with the velocity $v = (0, \ldots, 0,\, -1)$. That is, the flow is parallel to the $(d+1)$th coordinate axis and is directed ``downward''. If a particle hits the graph at a regular point, it is reflected according to the billiard law: {\it the angle of reflection equals the angle of incidence}. Taking into account that the law of reflection reads as $v^+ = v - 2\langle v,\, n \rangle n$, where $v^+$ is the velocity of the reflected particle, $n$ is the unit normal to the surface at the reflection point, and $\langle \cdot\,, \cdot \rangle$ indicates the scalar product, one easily calculates the velocity $v^+(x)$ of the particle after the reflection at a regular point $(x, u(x))$ of the graph. We have $n(x) = (-\nabla u(x), 1)/\sqrt{1 + |\nabla u(x)|^2}$ and
$$
v^+(x) = v - 2\langle v,\, n(x) \rangle n(x) = \frac{(-2\nabla u(x), \ 1 - |\nabla u(x)|^2)}{1 + |\nabla u(x)|^2}.
$$
This implies that after the reflection the particle moves along the ray
$$
(x - 2t \nabla u(x), \ u(x) + t(1 - |\nabla u(x)|^2)), \quad t \ge 0.
$$

We require that the ray lies above the graph of $u$, and therefore the particle does not hit $\graph(u)$ anymore. If this requirement is satisfied for all reflected rays, we say that $u$ satisfies the single impact condition (SIC). This condition can be stated analytically as follows: for all regular $x \in \Om$ and all $t \ge 0$ such that $x - 2t \nabla u(x) \in \bar\Om$,
\begin{equation}\label{eq2}
u(x - 2t \nabla u(x)) \le u(x) + t(1 - |\nabla u(x)|^2).
\end{equation}
A function $u$ is called {\it admissible}, if it satisfies conditions \eqref{eq1} and \eqref{eq2}.

\begin{remark}\label{zam1}
Suppose that $\nabla u(x) \ne 0$. Since $x - 2t \nabla u(x)$ lies in $\Om$ for $t = 0$ and in $\RRR^d \setminus \Om$ for $t$ sufficiently large, for a certain $t_0 > 0$ we have $x - 2t_0 \nabla u(x) \in \pl\Om$, and therefore \mbox{$u(x - 2t_0 \nabla u(x)) = 0$}. By \eqref{eq2} we have $u(x) + t_0 (1 - |\nabla u(x)|^2) \ge 0$; hence $|\nabla u(x)| < 1$. Thus we come to the following {\it necessary condition}:
\begin{quote}
If $u$ is admissible then $|\nabla u(x)| < 1$ at each regular point $x \in \Om$.
\end{quote}
That is, the inclination angle of $\graph(u)$ is everywhere smaller than $45^\circ$.
\end{remark}

\begin{center}
	\centering
		\includegraphics[scale=1]{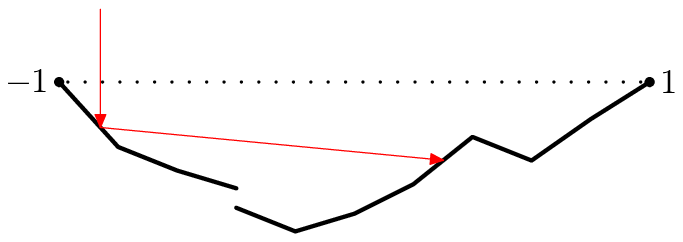}
		
	\figure The slope is greater than $45^\circ$
	\label{fig:slope greater than 45}
\end{center}

\begin{remark}\label{zam2}
One can easily derive the following {\it sufficient condition}:
\begin{quote}
If $u$ is continuous and $|\nabla u(x)| \le 1/\sqrt 3$ at each regular $x \in \Om$, then $u$ is admissible.
\end{quote}
It means that the inclination angle of $\graph(u)$ is everywhere smaller than or equal to $30^\circ$.

Indeed, if the condition is fulfilled then for $t \ge 0$
$$
|u(x - 2t \nabla u(x)) - u(x)| \le \sup_{\xi\in\Om} |\nabla u(\xi)| \cdot |2t \nabla u(x)| \le 2t/3 \text{ and } t(1 - |\nabla u(x)|^2) \ge 2t/3;
$$
hence \eqref{eq2} is true.
\end{remark}

\begin{center}
	\centering
		\includegraphics[scale=1]{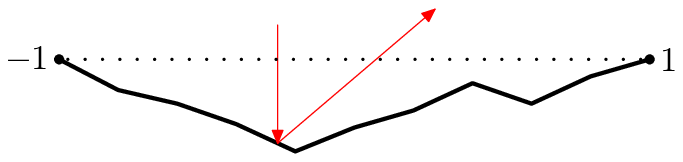}
		
	\figure A continuous function with the slope less than $30^\circ$.
	\label{fig:slope less than 30}
\end{center}

\begin{remark}\label{zam3}
One can have in mind the ``hollow'' in the $(d+1)$-dimensional half-space $\{ (x,z) \in \RRR^{d+1} : z \le 0 \}$ defined by
$$
\{ (x,z) \in \RRR^{d+1} :\, z \le u(x)\, \text{ if } x \in \Om \text{ and } z \le 0\, \text{ if } x \in \RRR^d \setminus \Om \}.
$$
The incident flow of particles is perpendicular to the hyperplane bounding the half-space. Each particle coming in the hollow hits it only once and then goes away. The particles do not mutually interact.
\end{remark}

We assume that the density $\rho$ of the flow is constant; then the force exerted by the flow on the hollow (usually called the {\it force of resistance}) equals $2\rho |\Om| (F_0,\, -F)$, where $|\Om|$ means the $d$-dimensional area of $\Om$,
$$
F_0 = F_0(u;\Om) = \frac{1}{|\Om|} \int_\Om \frac{\nabla u(x)}{1 + |\nabla u(x)|^2}\, dx,
 $$
and
 \begin{equation}\label{F}
 F = F(u;\Om) = \frac{1}{|\Om|} \int_\Om \frac{dx}{1 + |\nabla u(x)|^2}.
\end{equation}

We are concerned with the problem of minimizing the component of the force along the flow direction. Thus, the problem is as follows: given $\Om \subset \RRR^d$, find
\begin{equation}\label{eq3}
\inf_u F(u;\Om),
\end{equation}
where the infimum is taken over all admissible functions. The functional $F(u;\Om)$ is also called the {\it resistance} of $u$.

The problems of minimal resistance have a long history. The first problem of this kind was considered by Newton in his {\it Principia} (1687). In modern terms it can be formulated as minimizing the resistance $F(u;\Om)$ in the class of concave radially symmetric functions $u: \Om \to [0,\, h]$, where $\Om \subset \RRR^2$ is a unit circle and $h > 0$ is the parameter of the problem. (Note that all functions in this class satisfy SIC.) One can imagine a convex body bounded above by $\graph(u)$ and below by the horizontal plane $z = 0$; one is interested in finding the most streamlined shape of the body.

In the last two decades problems of minimal resistance in various classes of functions (usually wider than the class considered by Newton) have been studied \cite{BrFK,BB,Bsurvey,BFK,BK,CL1,CL2,Kawohl,LO,LP1}. The problem of minimizing the resistance of a hollow was first stated in \cite{CL3} in the particular case $d = 2$.

Since for an admissible function $u$ we have $|\nabla u(x)| < 1$ at all regular points $x$, we conclude that $1/2 < F(u;\Om) \le 1$, and therefore
$$
\inf_u F(u;\Om) \ge 1/2.
$$
It was proved in \cite{PkakSIAM} that for $d = 2$ and each $\Om \subset \RRR^2$,
$$
\inf_u F(u;\Om) = 1/2.
$$
This result can easily be extended to higher dimensions $d > 2$ (see the appendix). The only problem remaining open is $d = 1$.

The aim of this paper is twofold: first, solve the problem for $d = 1$, and second, solve the problem in arbitrary dimension in the class of radially symmetric functions $u$.

An open connected bounded set $\Om$ in the one-dimensional case is an interval; by a scaling transformation the problem can be reduced to the case when $\Om = (-1,\, 1)$.

Let us reformulate the { one-dimensional} problem for this specific case. We consider functions~$u$ on $[-1,\, 1]$ such that

(i) $u(-1) = 0 = u(1)$ and $u(x) < 0$ for $-1 < x < 1$;

(ii) the derivative $u'(x)$ exists and is continuous on an open full-measure subset of $[-1,\, 1]$;

(iii) $u$ satisfies SIC.\\
A function is called {\it admissible}, if it satisfies (i)--(iii).

Recall that $|u'(x)| < 1$ at each regular point $x$ of an admissible function $u$.

Single impact condition can be reformulated here as follows. We say that $u$ satisfies {\it forward (backward) SIC} at a regular point $x_0$, if $u'(x_0) < 0$ ($u'(x_0) > 0$) and the ray
$$
(x_0, u(x_0)) + \tau (-2u'(x_0),\ 1 - u'^2(x_0)), \ \tau \ge 0
$$
lies above the part of $\graph(u)$ to the right (to the left) of $x_0$.

Otherwise the forward and backward SIC can be expressed as follows. The function $u$ satisfies forward SIC at $x_0$, if $u'(x_0) < 0$ and
\begin{equation}\label{fSIC}
\frac{1 - u'^2(x_0)}{-2u'(x_0)} \ge \frac{u(x) - u(x_0)}{x - x_0} \quad \text{for all } \ x_0 < x \le 1,
\end{equation}
and satisfies backward SIC at $x_0$, if $u'(x_0) > 0$ and
\begin{equation}\label{bSIC}
\frac{1 - u'^2(x_0)}{2u'(x_0)} \ge \frac{u(x) - u(x_0)}{x_0 - x} \quad \text{for all } \ -1 \le x < x_0.
\end{equation}
We say that $u$ satisfies SIC, if at each regular point with nonzero derivative it satisfies either forward or backward SIC.

The resistance of the function $u$ equals
$$
F(u) = \frac 12 \int_{-1}^1 \frac{dx}{1 + u'^2(x)}.
$$


\begin{proposition}\label{prop1}
The infimum of $F$ in the class of admissible functions $u : [-1,\, 1] \to \RRR$ is attained at~$u_0$, where $u_0(x) = ((|x| + 1)^2 - 4)/4$,
and is equal to $F(u_0) = \pi/2 - 2\arctan(1/2) \approx 0.6435$.
\end{proposition}

The graph of $u_0$ (formed by two mutually symmetric arcs of parabolas with foci at $(1,0)$ and $(-1,0)$) is shown in Fig.~\ref{fig:u0}.

\begin{center}
\parbox{13cm}{\begin{picture}(0,130)
\rput(2.2,3.4){
\scalebox{0.8}{
 \psecurve[linewidth=1pt,linecolor=blue]
 (4,-4.2)(5,-3.75)(6,-3.2)(7,-2.55)(8,-1.8)(9,-0.95)(10,0)(11,1.05)
  \psecurve[linewidth=1pt,linecolor=blue]
 (6,-4.2)(5,-3.75)(4,-3.2)(3,-2.55)(2,-1.8)(1,-0.95)(0,0)(-1,1.05)
\rput(10,0.35){\scalebox{1.2}{1}}
\rput(0,0.35){\scalebox{1.2}{$-1$}}
\rput(10.9,-0.3){\scalebox{1.2}{$x$}}
\psline[linewidth=0.4pt,arrows=->,arrowscale=1.5](5,-4.5)(5,1.5)
\psline[linewidth=0.4pt,arrows=->,arrowscale=1.5](-0.7,0)(11,0)
\psline[linewidth=0.6pt,linestyle=dotted](5,0)(5,-3.75)
   \psline[linewidth=0.6pt,arrows=->,arrowscale=1.8,linecolor=red](2,1)(2,-0.5)
   \psline[linewidth=0.6pt,arrows=->,arrowscale=1.8,linecolor=red](2,-0.5)(2,-1.8)(11.6,0.36)
   \psline[linewidth=0.6pt,arrows=->,arrowscale=1.8,linecolor=red](4,1)(4,-0.5)
   \psline[linewidth=0.6pt,arrows=->,arrowscale=1.8,linecolor=red](4,-0.5)(4,-3.2)(11.5,0.8)
}}
\end{picture}}
\vskip 0.3cm
\figure \label{fig:u0} The optimal function $u_0$ and two particles of the vertical flow reflected by $\graph(u_0)$.
\end{center}

Thus, one comes to the following result:
\begin{quote}
The minimal resistance of a hollow depends only on the dimension $d$ and does not depend on $\Om$, and equals $0.6435$, if $d = 1$, and $0.5$, if $d \ge 2$.
\end{quote}


Now consider the $d$-dimensional case where the function $u$ is radial and correspondingly the domain~$\Om$ is a ball centered at the origin. By a scaling transformation of $\RRR^d$ the problem can be reduced to the case when $\Om$ is the unit ball $|x| < 1$ and $u(x) = \phi(|x|)$, with $\phi(\xi) < 0$ for $0 \le \xi < 1$ and $\phi(\xi) = 0$ for $\xi \ge 1$.

\begin{proposition}\label{prop2-0}
The infimum of $F(u; \Om)$ \eqref{F} in the class of admissible radial functions $u(x) = \phi(|x|)$ is attained at the function $u(x) = ((|x| + 1)^2 - 4)/4$,\, $x \in \Om$. In particular, in the case $d = 2$ the optimal hollow is obtained by rotating the graph of $u_0$ in Fig.~\ref{fig:u0} about the vertical axis.
\end{proposition}

Note that in the radial case we have $|\nabla u(x)| = \phi'(\xi)$. Let $\omega_{d-1}$ be the volume of the $(d-1)$-dimensional unit sphere.
Since the volume of the $d$-dimensional unit ball equals $\omega_{d-1}/d$, we get that the multiple integral $F(u; \Omega)$ in \eqref{F} can be written in terms of $\phi$ as
\[
F_d(\phi) =\frac{d}{\omega_{d-1}} \int_0^1 \frac{\omega_{d-1}\xi^{d-1}}{1 + \phi'^2(\xi)}\, d(\xi) = \int_0^1 \frac{d\, \xi^{d-1}}{1 + \phi'^2(\xi)}\, d\xi.
\]

It is convenient to assume that $\phi$ is even and defined on $[-1,\, 1]$; then condition SIC for~$\phi$ is equivalent to condition SIC for the corresponding radial function $u$. Indeed, any vertical central cross section of $\graph(u)$ is congruent to $\graph(\phi)$, and the trajectory of a particle in $\RRR^{d+1}$ reflected by $\graph(u)$ lies in such a cross section and is congruent to a trajectory in $\RRR^2$ reflected by $\graph(\phi)$. If all trajectories in the $d$-dimensional radial case have at most one reflection, the same is true for trajectories in the 1-dimensional case, and vice versa.

Thus, Proposition \ref{prop2-0} can be reformulated in the new terms as follows.




\vspace{2mm}

\noindent{\bf Proposition \ref{prop2-0}'.}
{\it The infimum of $F_d$ in the class of even admissible functions $\phi : [-1,\, 1] \to \RRR$ is attained at $\phi(\xi) = u_0(\xi) = ((|\xi| + 1)^2 - 4)/4$. }
\vspace{2mm}

\begin{remark}\label{zam5}
Proposition \ref{prop2-0} states that the minimizer coincides with the minimizer of $F$ in Proposition~\ref{prop1} (and therefore does not depend on $d$). The minimal resistance $m_d = F_d(u_0)$ is always greater than~$0.5$. The values $m_d$ are monotone decreasing from $m_1 \approx 0.6435$ to $0.5$ as $d \to \infty$. Indeed, since for $0 < \ve < 1$ the ratio
$$
\frac{\text{volume of the ball with radius $1-\ve$}}{\text{volume of the unit ball}}
$$
goes to zero as $d \to \infty$, we conclude that almost all volume is concentrated near the boundary where the specific resistance is approximately $1/2$. In the limit $d \to \infty$,\, $\ve \to 0$ we obtain that the resistance goes to $1/2$.
%
\end{remark}




Both Propositions \ref{prop1} and \ref{prop2-0}' are simple consequences of the following theorem.

\begin{theorem}\label{t1}
Let $f$ be a non-negative even differentiable function on $[-1,\, 1]$ monotone non-decreasing on $[0,\, 1]$, and let
\begin{equation}\label{FFF}
\FFF(u) = \int_{-1}^1 \frac{f(x)}{1 + u'^2(x)}\, dx.
\end{equation}
Then the infimum of $\FFF$ in the class of even admissible functions $u: [-1,\, 1] \to \RRR$ is attained at $u_0(x) = ((|x| + 1)^2 - 4)/4$.
\end{theorem}

Indeed, substituting $f(x) = \frac{d}{2}\, |x|^{d-1}$ in \eqref{FFF} and using that $\FFF(u) = F_d(u)$, one immediately obtains Proposition~\ref{prop2-0}'. Further, Proposition \ref{prop1} is a direct consequence of Theorem \ref{t1} (with $f(x) \equiv 1/2$) and the following lemma.

\begin{center}
		(a) \includegraphics[scale=1]{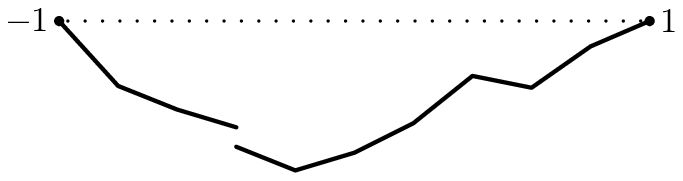}\\
		\vskip 0.4cm
		
		(b) \includegraphics[scale=1]{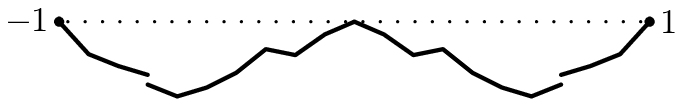}\\
	\figure \label{fig:symmetrization of function}
	The function before (a) and after (b) symmetrization.
\end{center}

\begin{lemma}\label{l0}
The infima of $F$ in the class of admissible functions and in the narrower class of {\rm even} admissible functions coincide.
\end{lemma}

\begin{proof}
It suffices to show that for any admissible function $u$ there is an {\it even} admissible function $\tilde u$ with the same resistance, $F(\tilde u) = F(u)$. Let $\tilde u(x) = \frac 12\, u(2|x| - 1)$. The graph of $\tilde u$ is obtained by putting together two smaller copies of $\graph(u)$ mutually symmetric with respect to the vertical axis (Fig.~\ref{fig:symmetrization of function}). Thus, $\tilde u$ satisfies SIC and is even, and
$
F(\tilde u) = F(u)$. Lemma \ref{l0} is proved.
\end{proof}

\section{Proof of Theorem \ref{t1}}

We say that $u$ satisfies {\it strong SIC}, if there exists $\del > 0$ such that $u'(x_0) < 0$ implies the inequality
\begin{equation}\label{strongSIC1}
\frac{1 - u'^2(x_0)}{-2u'(x_0)} \ge \frac{u(x) - u(x_0)}{x - x_0} + \del \quad \text{for all } \ x_0 < x \le 1
\end{equation}
and $u'(x_0) > 0$ implies the inequality
\begin{equation}\label{strongSIC2}
\frac{1 - u'^2(x_0)}{2u'(x_0)} \ge \frac{u(x) - u(x_0)}{x_0 - x} + \del \quad \text{for all } \ -1 \le x < x_0.
\end{equation}

\begin{lemma}\label{l1-1}
For any even admissible function $u$ there is an even admissible function $\hat u$ {\rm satisfying strong SIC} and such that $\FFF(\hat u) < \FFF(u) + \ve$ and $\hat u(x) < 0$ for $-1 < x < 1$.
\end{lemma}

\begin{proof}
Take an open full-measure set $\UUU \subset [-1,\, 1]$ where $u$ is continuous, and let $S_1 = \{ x \in \UUU : u(x) = 0\}$ and $S_2 = (-1,\, 1) \setminus \UUU$. Fix a value $\om > 0$ (to be chosen later) and choose a negative value $c$ such that the measure of the set $B_c = \{ x \in \UUU : u(x) \le c \}$ is smaller than $\om$. Then cover the sets $S_1, \, S_2,\, B_c$, respectively, by open sets $\OOO_1,\, \OOO_2,\, \OOO_3$ such that $|\OOO_1 \setminus S_1| < \om, \ |\OOO_2| < \om$, and $|\OOO_3| < \om$, where $|\cdots|$ means Lebesgue measure on the line.

Fix a value $0 < q < 1$ (to be chosen later) and define
$$
\hat u(x) = \left\{
\begin{array}{ll}
q u(x), & \text{if } x \in [-1,\, 1] \setminus (\bar\OOO_1 \cup \bar\OOO_2 \cup \bar\OOO_3)\\
c, & \text{otherwise}.
\end{array}
\right.
$$
Notice that $\hat u'(x) = 0$ for all $x \in \OOO_1 \cup \OOO_2 \cup \OOO_3$. Let us show that $\hat u$ satisfies strong SIC.

Put $\del = (1 - q^2)/2$ and take into account that $|\hat u'(x_0)| < 1$. If $\hat u'(x_0) < 0$ then $\hat u(x) = qu(x)$ and $\hat u'(x) = qu'(x)$, and by \eqref{fSIC}
\begin{multline}\label{hatsic1}
\frac{1 - \hat u'^2(x_0)}{-2\hat u'(x_0)} = \frac{1 - q^2}{-2\hat u'(x_0)} + \frac{q^2 - \hat u'^2(x_0)}{-2\hat u'(x_0)}
\ge \del + q\, \frac{1 - u'^2(x_0)}{-2u'(x_0)}\\
\ge \del + q\, \frac{u(x) - u(x_0)}{x - x_0} = \del + \frac{\hat u(x) - \hat u(x_0)}{x - x_0} \quad \text{for all} \ \, x_0 \le x \le 1,
\end{multline}
and relation \eqref{strongSIC1} for $\hat u$ is proved. The proof of \eqref{strongSIC2} for $\hat u$ in the case $\hat u'(x_0) > 0$ is analogous. Thus, $\hat u$ satisfies strong SIC.

Further, we have
$$
\int_{\bar\OOO_1 \cup \bar\OOO_2 \cup \bar\OOO_3} \left( \frac{f(x)}{1 + \hat u'^2(x)} - \frac{f(x)}{1 + u'^2(x)} \right) dx < 3\om\, f(1).
$$
Taking $\om < \ve/(6f(1))$, we ensure that this expression is smaller than $\ve/2$. The difference
$$
\int_{[-1,\, 1] \setminus \bar\OOO_1 \cup \bar\OOO_2 \cup \bar\OOO_3} \left( \frac{f(x)}{1 + q^2 u'^2(x)} - \frac{f(x)}{1 + u'^2(x)} \right) dx
$$
is smaller than $\ve/2$ for $1-q$ sufficiently small. Thus, for the chosen values of $\om$ and $q$ we get $\FFF(\hat u) - \FFF(u) < \ve$.
\end{proof}

\begin{lemma}\label{l1-2}
For any  $\ve > 0$ and any even admissible function $u$ satisfying strong SIC and such that $u(x) < 0$ for $-1 < x < 1$, there is a piecewise linear even admissible function $\tilde u$ such that $\FFF(\tilde u) < \FFF(u) + \ve$.
\end{lemma}

\begin{proof}
Choose a finite set of disjoint closed segments $I_i \subset (-1,\, 1)$ such that $u'$ is continuous on each $I_i$ and $|\cup_i I_i| > 2 - \ve/(2 \max_{[-1,\,1]} f(x))$. We additionally require that this set of segments is symmetric with respect to the point $x = 0$. Take a small $\s > 0$ (to be chosen later) and find an odd piecewise constant function $w(x)$ approximating $u'(x)$ on $\cup_i I_i$, so that $|u'(x)| - \s \le |w(x)| \le |u'(x)|$ and $w(x) u'(x) \ge 0$ for all $x \in \cup_i I_i$.

Define the piecewise linear even function $\tilde u$ as follows: $\tilde u(x)$ is a primitive of $w(x)$ on each $I_i$ and coincides with $u(x)$ at the center of each segment $I_i$. Further, $\tilde u$ coincides with $\inf_{x\in\cup_i I_i} \tilde u(x)$ on $(-1,\, 1) \setminus (\cup_i I_i)$, and $\tilde u(-1) = 0 = \tilde u(1)$. Making $\s$ sufficiently small, one can assure that $\tilde u(x)$ is negative for $-1 < x < 1$.

Take $\s = \s(\del)$ so small that for all $x_0 \in \cup_i I_i$ and $x \in \cup_i I_i \cup \{ -1,\, 1 \}$,
$$
|(\tilde u(x) - \tilde u(x_0)) - (u(x) - u(x_0))| \le \del |x - x_0|.
$$
(It suffices to take $\s = \del d/2$, where $d$ is the length of the smallest connected component of $(-1,\, 1) \setminus (\cup_i I_i)$.) Then for $x \ne x_0$ we have
\begin{equation}\label{eq31}
\left| \frac{\tilde u(x) - \tilde u(x_0)}{x - x_0} - \frac{u(x) - u(x_0)}{x - x_0} \right| \le \del.
\end{equation}
Assume that $\tilde u(x_0) < 0$. Since $0 < -\tilde u'(x_0) \le -u'(x_0) < 0$, we have
$$
\frac{1 - \tilde u'^2(x_0)}{-2\tilde u'(x_0)} \ge \frac{1 - u'^2(x_0)}{-2u'(x_0)},
$$
and taking into account \eqref{hatsic1} and \eqref{eq31}, one concludes that \eqref{fSIC} for $\tilde u$ is satisfied. Relation \eqref{bSIC} in the case $\tilde u(x_0) > 0$ is proved analogously. Notice that substituting $x = -1$ and $x = 1$ in these relations, one obtains $|\tilde u'(x_0)| < 1$.

For $x_0 \in \cup_i I_i$ and $x \in (-1,\, 1) \setminus (\cup_i I_i)$,\, $\tilde u(x) - \tilde u(x_0)$ is negative, and therefore relations \eqref{fSIC} and \eqref{bSIC} for $\tilde u$ are satisfied. Finally, for $x_0 \in (-1,\, 1) \setminus (\cup_i I_i)$ we have $\tilde u'(x_0) = 0$. This implies that $\tilde u$ satisfies SIC, and therefore is admissible.

One has
$$
\int_{[-1,\, 1] \setminus (\cup_i I_i)} f(x)\, dx -
\int_{[-1,\, 1] \setminus (\cup_i I_i)} \frac{f(x)}{1 + u'^2(x)}\, dx \le |(-1,\, 1) \setminus (\cup_i I_i)|\cdot \max_{[-1,\,1]} f(x)) < \ve/2,
$$
and taking $\s$ sufficiently small one obtains
$$
\Big| \int_{\cup_i I_i} \frac{f(x)}{1 + \tilde u'^2(x)}\, dx - \int_{\cup_i I_i} \frac{f(x)}{1 + u'^2(x)}\, dx \Big| < \ve/2.
$$
Thus, $\FFF(\tilde u) < \FFF(u) + \ve$. Lemma \ref{l1-2} is proved.
\end{proof}

The following corollary is a direct consequence of Lemmas \ref{l1-1} and \ref{l1-2}.

\begin{corollary}\label{cor}
For any admissible even function $u$ and $\ve > 0$ there is a piecewise linear admissible even function $\tilde u$ such that $\FFF(\tilde u) < \FFF(u) + \ve$.
\end{corollary}

\begin{proof}
Indeed, by Lemma \ref{l1-1} there is a function $\hat u(x)$ satisfying strong SIC and such that $\hat u(x) < 0$ for $-1 < x < 1$. Then by Lemma \ref{l1-2} there is a piecewise linear even admissible function $\bar{u}$ such that $\FFF(\tilde u) < \FFF(\tilde u) + \ve/2$. The corollary is proved.
\end{proof}


\begin{lemma}\label{l2}
For each piecewise linear even admissible function $u$ there exists a convex even admissible function $\tilde u$ continuous at $x = 1$ (and therefore on $[-1,\, 1]$) such that $\FFF(\tilde u) \le \FFF(u)$.
\end{lemma}

\begin{proof}
The generalized graph of $u$ is a broken line; let $p_0 = (-1,0),\, p_0' = (-1, u(0^+)),\, p_1,\, p_1', \ldots,\, p_{n-1},$ $p_{n-1}',\, p_n = (1, u(1^-)),\, p_n' = (1,0)$ be its vertices, with $p_i = (x_i, y_i)$ and $p_i' = (x_i, y_i')$. Thus, the points $x_0 = -1,\, x_1, \ldots,$ $x_{n-1},\, x_n = 1$ are singular points of $u$; in general $u$ has discontinuities at these points, and $y_i = u(x_i-0)$ and $y_i' = u(x_i+0)$. The part of the graph of $u$ corresponding to an interval $(x_{i-1},\, x_i)$ is the segment with the endpoints $p_{i-1}'$ and $p_i$.

Denote $\vec v_i = p_i - p_{i-1}'$ and $\vec w_i = p_i' - p_i$; then the generalized graph of $u$ is represented by the sequence of vectors $\vec w_0,\, \vec v_1,\, \vec w_1, \ldots, \vec v_{n},\, \vec w_{n}$. The vectors $\vec v_i$ are directed to the right, and the vectors $\vec w_i$ are directed vertically (upward or downward). One obviously has $\vec w_{0} + \vec v_1 + \vec w_1 + \ldots +\vec v_n + \vec w_{n} = (2,0)$.

Since $u$ is even, $p_i'$ is symmetric to $p_{n-i}$ with respect to the vertical line $x = 0\ (i = 0,\, 1, \ldots, n)$. Without loss of generality we assume that $n$ is even. Otherwise the middle segment of the broken line (which is necessarily horizontal and symmetric with respect to the line $x = 1/2$) is represented as the union of two segments of equal length.

Denote by $\kap_i = \frac{y_i - y_{i-1}'}{x_i - x_{i-1}}$ the slope ratio of $\vec v_i$, and place the ratios in the increasing order,
$$
\kap_{i_1} \le \ldots \le \kap_{i_{n/2}} \le 0 \le \kap_{i_{{n/2}+1}} \le \ldots \le \kap_{i_n}.
$$
One obviously has $\kap_{i_j} = -\kap_{i_{n+1-j}}$, and the vectors $\vec v_{i_{j}}$ and $\vec v_{i_{n+1-j}}$ are symmetric with respect to the horizontal line for all $j$.

In what follows we use the notation $[\vec w_{0},\, \vec v_1,\, \vec w_1, \ldots, \vec v_n,\, \vec w_{n}]_{j_1,\ldots,j_m}$ for the sequence obtained from $\vec w_{0},\, \vec v_1,\, \vec w_1, \ldots, \vec v_n,\, \vec w_{n}$ by (a) removing the elements $\vec v_{j_1}, \ldots, \vec v_{j_m}$ and (b) substituting each maximal subsequence of consecutive vectors $\vec w_\al,\, \vec w_\bt, \ldots, \vec w_\gam$ in the resulting sequence by their sum $\vec w_\al + \vec w_\bt + \ldots + \vec w_\gam$.

The oriented broken line with the initial vertex at the point $(-1, 0)$ and with the sequence of edges
$$
\vec v_{i_1}, \ldots, \vec v_{i_{n/2}},\, [\vec w_{0},\, \vec v_1,\, \vec w_1, \ldots, \vec v_n,\, \vec w_{n}]_{i_1,\ldots,i_{n/2}}
$$
is the generalized graph of a function; let it be denoted by $\tilde u_1$.
One easily verifies that $\tilde u_1$ is negative on $(-1,\, 1)$, continuous at $-1$, and equal to zero at $-1$ and 1.
Since the graph of $\tilde u_1$ is obtained by permutation of parts (segments) of the graph of $u$, one has $F(\tilde u_1) = F(u)$.
Since the slopes of the vectors $\vec v_{i_1}, \ldots, \vec v_{i_{n/2}}$ are non-decreasing, $\tilde u_1$ is convex on $[-1,\, 0]$.

The function $\tilde u$ is defined by $\tilde u(x) = \tilde u_1(-|x|)$. It is even and convex, and is continuous at $x = 1$. The graph of $\tilde u$ is formed by the sequence of vectors with non-decreasing slopes $\vec v_{i_1}, \ldots, \vec v_{i_{n}}$. Introduce the notation
$$
g(x) = \frac{1}{1 + u'^2(x)} \quad \text{and} \quad g^*(x) = \frac{1}{1 + \tilde u'^2(x)}.
$$
Obviously, the function $g^*$ is even and monotone non-increasing on $[0,\, 1]$. Let $\Del_i$ be the $x$-component of the vector $\vec v_i$; that is, $\Del_i = x_i - x_{i-1}$. Notice that the function $f$ coincides with its symmetric increasing rearrangement, while $g^*$ is the symmetric decreasing rearrangement of $g$. Indeed, for
$$\frac{1}{1 + \kappa_{i_{j}}^2} \le c < \frac{1}{1 + \kappa_{i_{j+1}}^2} \quad \quad (1 \le j \le n/2 - 1)$$
the standard linear measure of the set $\{ x : g^*(x) \le c \}$ coincides with the measure of the set $\{ x : g(x) \le c \}$ and equals $\Del_{i_1} + \ldots + \Del_{i_{j}} + \Del_{i_{n+1-j}} + \ldots \Del_{i_n}$. For $c < 1/(1 + \kappa_{i_{j}}^2)$ both sets are empty, and for $c \ge 1/(1 + \kappa_{i_{n/2}}^2)$ both sets coincide with $[-1,\, 1]$. Then by the 1-dimensional Hardy-Littlewood inequality one has
$$
\int_{-1}^1 f(x) g(x)\, dx \ge \int_{-1}^1 f(x) g^*(x)\, dx,
$$
and so, inequality $\FFF(u) \ge \FFF(\tilde u)$ is proved.

It remains to check condition SIC for $\tilde u$. We shall first prove that $\tilde u_1$ satisfies forward SIC.

The proof is done by (finite) induction. Define successive permutations in the sequence of vectors, starting from $\vec w_{0},\, \vec v_1,\, \vec w_1, \ldots, \vec v_n,\, \vec w_{n}$. At the 1st step put the (most declining) vector $\vec v_{i_1}$ on the first place, thus obtaining the sequence $\vec v_{i_1},\, [\vec w_{0},\, \vec v_1,\, \vec w_1, \ldots, \vec v_n,\, \vec w_{n}]_{i_1}$. At the 2nd step we put the vector $\vec v_{i_2}$ on the second place, thus obtaining the sequence $\vec v_{i_1},\, \vec v_{i_2},\, [\vec w_{0},\, \vec v_1,\, \vec w_1, \ldots,\, \vec v_n,\, \vec w_{n}]_{i_1,\, i_2}$. Repeating the procedure several times, at the $(n/2)$th step we finally obtain the sequence $\vec v_{i_1}, \ldots,$ $\vec v_{i_{n/2}},\, [\vec w_{0},\, \vec v_1,\, \vec w_1, \ldots, \vec v_n,\, \vec w_{n}]_{i_1,\ldots,i_{n/2}}$.

Let the function obtained at the $k$th step be denoted by $u_k$. Thus, one obtains the finite sequence of functions $u_0 = u,\, u_1, \ldots, u_{n/2-1},\, u_{n/2} = \tilde u_1$. By the hypothesis of the lemma, $u$ satisfies forward SIC. Assume that $u_{k-1} \, (1 \le k \le n/2)$ satisfies forward SIC; we are going to prove that $u_{k}$ also does.

The sequences representing $u_{k-1}$ and $u_{k}$ can be written down as
$$
\vec v_{i_1}, \ldots, \vec v_{i_{k-1}}; \ [\vec w_{0},\, \vec v_1,\, \vec w_1 \ldots, \vec w_{i_{k}-2},\, \vec v_{i_{k}-1}]_{i_1,\ldots,i_{k-1}},
\vec w_{i_{k}-1},\, \vec v_{i_{k}},\, \vec w_{i_{k}};
$$
$$
[\vec v_{i_{k}+1},\, \vec w_{i_{k}+1} \ldots, \vec v_n,\, \vec w_{n}]_{i_1,\ldots,i_{k-1}}
$$
and
$$
\vec v_{i_1}, \ldots, \vec v_{i_{k-1}}; \ \vec v_{i_{k}}, [\vec w_{0},\, \vec v_1,\, \vec w_1 \ldots, \vec w_{i_{k}-2},\, \vec v_{i_{k}-1}]_{i_1,\ldots,i_{k-1}},
\vec w_{i_{k}-1} + \vec w_{i_{k}};
$$
$$
[\vec v_{i_{k}+1},\, \vec w_{i_{k}+1} \ldots, \vec v_n,\, \vec w_{n}]_{i_1,\ldots,i_{k-1}}.
$$
We see that both sequences can be divided into 3 parts. At the $k$th step of induction the initial and final parts remain unchanged, and the subsequence $[\vec w_{0},\, \vec v_1,\, \vec w_1 \ldots, \vec w_{i_{k}-2},\, \vec v_{i_{k}-1}]_{i_1,\ldots,i_{k-1}}$ is exchanged with the vector $\vec v_{i_{k}}$ in the middle part.

\begin{center}
\parbox{13cm}{\begin{picture}(0,250)

\rput(2.5,6.8){
\scalebox{1.2}{
\rput(-1.2,0.7){\scalebox{1.5}{$u_{k-1}$}}
\rput(-2,-1.5){\scalebox{0.85}{(a)}}
  \psline[linewidth=0.67pt,linecolor=red,arrows=->,arrowscale=1.5](1,0.7)(1,-0.8)
  \psline[linewidth=0.67pt,linecolor=red,arrows=->,arrowscale=1.5](1,-0.8)(1,-0.9)(3.2,-0.405)
    \rput(3.45,-0.4){\scalebox{0.83}{\red \rput(0.1,0.13){$\bigcirc$}1}}
  \psline[linewidth=0.67pt,linecolor=red,arrows=->,arrowscale=1.5](3,0.7)(3,-1.2)
  \psline[linewidth=0.67pt,linecolor=red,arrows=->,arrowscale=1.5](3,-1.2)(3,-1.32)(5,-0.25)
    \rput(5.25,-0.25){\scalebox{0.83}{\red \rput(0.1,0.13){$\bigcirc$}2}}
\rput(1.18,-0.65){\scalebox{0.8}{A}}
\rput(3.18,-1){\scalebox{0.8}{B}}
\rput(2.46,-0.78){\scalebox{0.8}{C}}
\rput(0.03,-0.33){\scalebox{0.8}{$\vec v_{i_1}$}}
\rput(0.6,-0.9){\scalebox{0.8}{$\vec v_{i_{k-1}}$}}
\rput(2.7,-1.45){\scalebox{0.8}{$\vec v_{i_{k}}$}}
\psdots(1,-0.9)(2.5,-1.02)(3,-1.32)
\psdots[dotsize=2pt](0,0)(8,0)(0.5,-0.5)
  \psline[linestyle=dotted,linewidth=0.6pt,dotsep=2.2pt](1,0)(1,-0.9)
  \psline[linestyle=dotted,linewidth=0.6pt,dotsep=2.2pt](3,0)(3,-1.32)
  \rput(1.1,0.2){\scalebox{0.83}{$a$}}  \rput(3.1,0.2){\scalebox{0.83}{$b$}}
  \rput(0.05,0.25){\scalebox{0.8}{$-1$}}  \rput(8.1,0.25){\scalebox{0.83}{1}}
\psline[linestyle=dotted,linewidth=0.6pt,dotsep=2.2pt](0,0)(8,0)
\psline[linewidth=0.67pt](0,0)(0.5,-0.5)(1,-0.9)
\psline[linewidth=0.67pt](2.5,-1.02)(3,-1.32)
   \rput(0,-0.2)
   {\psline[linewidth=0.67pt](1,-0.9)(1.3,-1)(1.6,-0.95)(1.9,-1.05)(2.2,-0.93)(2.5,-0.91)}
\psline[linestyle=dashed,linewidth=0.67pt](3,-1.45)(3.2,-1.45)(3.6,-1.15)(4.2,-1.25)(4.9,-0.9)(5.6,-0.9)(6.4,-0.5)(7,-0.6)(8,0)
}}

\rput(2.5,2.3){
\scalebox{1.2}{
\rput(-1,0.2){\scalebox{1.5}{$u_{k}$}}
\rput(-2,-1.5){\scalebox{0.85}{(b)}}
  \psline[linewidth=0.67pt,linestyle=dashed,linecolor=red,arrows=->,arrowscale=1.5](1.5,-1.2)(3.5,-0.75)
  \rput(3.7,-0.7){\scalebox{0.83}{\red \rput(0.1,0.13){$\bigcirc$}1}}
  \psline[linewidth=0.67pt,linecolor=red,linestyle=dashed,arrows=->,arrowscale=1.5](3,-1.32)(5,-0.25)
    \rput(5.2,-0.25){\scalebox{0.83}{\red \rput(0.1,0.13){$\bigcirc$}2}}
  \psline[linewidth=0.67pt,linecolor=red,arrows=->,arrowscale=1.5](1.5,0.7)(1.5,-1.1)
  \psline[linewidth=0.67pt,linecolor=red,arrows=->,arrowscale=1.5](1.5,-1.1)(1.5,-1.2)(3.5,-0.13)
\rput(1.22,-0.73){\scalebox{0.8}{A}}
\rput(2.8,-1.15){\scalebox{0.8}{B}}
\rput(1.35,-1.35){\scalebox{0.8}{D}}
\rput(0.45,-0.84){\scalebox{0.8}{$\vec v_{i_{k-1}}$}}
\rput(1.05,-1.15){\scalebox{0.8}{$\vec v_{i_{k}}$}}
\psdots(1,-0.9)(3,-1.32)(1.5,-1.2)
\psdots[dotsize=2pt](0,0)(8,0)(0.5,-0.5)
  \psline[linestyle=dotted,linewidth=0.6pt,dotsep=2.2pt](1,0)(1,-0.9)
  \psline[linestyle=dotted,linewidth=0.6pt,dotsep=2.2pt](3,0)(3,-1.32)
  \rput(1,0.2){\scalebox{0.83}{$a$}}  \rput(3,0.2){\scalebox{0.83}{$b$}}
  \rput(0.05,0.25){\scalebox{0.8}{$-1$}}  \rput(8.1,0.25){\scalebox{0.83}{1}}
\psline[linestyle=dotted,linewidth=0.6pt,dotsep=2pt](0,0)(8,0)
\psline[linewidth=0.67pt](0,0)(0.5,-0.5)(1,-0.9)(1.5,-1.2)
\rput(0.5,-0.3){
   \rput(-0.05,-0.2)
   {\psline[linewidth=0.67pt](1,-0.9)(1.3,-1)(1.6,-0.95)(1.9,-1.05)(2.2,-0.93)(2.5,-0.91)}}
\psline[linestyle=dashed,linewidth=0.67pt](3,-1.45)(3.2,-1.45)(3.6,-1.15)(4.2,-1.25)(4.9,-0.9)(5.6,-0.9)(6.4,-0.5)(7,-0.6)(8,0)
}}
\end{picture}}

\figure \label{fig:permutation}
Permutation of the sequence of edges.
\end{center}

The interval $[a, b] \subset [-1,\, 1]$ (see Fig.~\ref{fig:permutation}) corresponds to the subsequence $[\vec w_{0},\, \vec v_1,\, \vec w_1 \ldots, \vec w_{i_{k}-2},$ $\vec v_{i_{k}-1}]_{i_1,\ldots,i_{k-1}},\, \vec w_{i_{k}-1},\, \vec v_{i_{k}},\, \vec w_{i_{k}}$ in figure (a), and to the permuted subsequence $\vec v_{i_{k}},\, [\vec w_{0},\, \vec v_1,\, \vec w_1 \ldots,$ $\vec w_{i_{k}-2},$ $\vec v_{i_{k}-1}]_{i_1,\ldots,i_{k-1}}$, $\vec w_{i_{k}-1} + \vec w_{i_{k}}$ in figure (b). As a result of permutation, the function remains unchanged outside $[a, b]$; that is, $u_k(x) = u_{k-1}(x)$ for $x \in [-1, a] \cup [b, 1]$. Since $u_{k-1}$ satisfies forward SIC, we conclude that forward SIC is satisfied for the function $u_{k}$ on the interval $[b, 1]$. It remains to check the intervals $[a, b]$ and $[-1,\, a]$.

The broken line corresponding to the subsequence $[\vec w_{0},\, \vec v_1,\, \vec w_1 \ldots, \vec w_{i_{k}-2}$, $\vec v_{i_{k}-1},\, \vec w_{i_{k}-1}]_{i_1,\ldots,i_{k-1}}$ and the vector $\vec v_{i_{k}}$ are denoted, respectively, by $AC$ and $CB$ in Fig.~\ref{fig:permutation}\,(a), and by $DB$ and $AD$ in Fig.~\ref{fig:permutation}\,(b). We are going to show that forward SIC is satisfied, first, on the segment $AD$, and second, on the broken line $DB$ corresponding to the function $u_k$ (see Fig.~\ref{fig:permutation}\.(b)).

The ray reflected from $\vec v_{i_{k-1}}$ at the point $A$ in figure (a) is denoted by \rput(0.1,0.13){$\bigcirc$}1\,. (If $k = 1$ then by definition the ray is horizontal with the vertex at $(-1,0)$ and is directed to the right.) It is situated above the generalized graph of $u_{k-1}$, and therefore above the point $C$. The ray with the same direction and with the vertex at $D$ in figure (b) is also denoted by \rput(0.1,0.13){$\bigcirc$}1\,. Since the broken line $DB$ is a translation of $AC$, the ray is situated above the broken line $DB$, and in particular above~$B$.

The ray reflected from $\vec v_{i_{k}}$ at the point $B$ in figure (a) and the same ray in figure (b) are denoted by \rput(0.1,0.13){$\bigcirc$}2\,. They are situated above the graphs of the functions (which are identical to the right of $B$ in figures (a) and (b)).

The slope of $\vec v_{i_{k}}$ is less than that of $\vec v_{i_{k-1}}$, and so, the ray reflected from $\vec v_{i_{k}}$ at $D$ in figure (b) is situated above the ray \rput(0.1,0.13){$\bigcirc$}1\,, and therefore above $B$. This ray is parallel to the ray \rput(0.1,0.13){$\bigcirc$}2\,, and hence is situated above it. Thus, this ray, and therefore any other ray reflected from $\vec v_{i_{k}}$ in figure (b), are situated above the graph. Thus, forward SIC for $AD$ is proved.

Further, each ray reflected at a point with negative slope between $D$ and $B$ in figure (b), goes above the part $DB$ of the graph, and therefore above $B$, and its slope is greater than that of the ray \rput(0.1,0.13){$\bigcirc$}2\,. Therefore it is situated above the rest of the graph, and forward SIC for $DB$ is also proved.

Each ray reflected by the broken line $[-1]A$ in figure (a) lies above the broken line $AC$, the segment $CB$, and the broken line $B[1]$ (the dashed line). Now consider the same ray in figure (b). Its slope is positive, it lies above the point $A$, and the slope of the segment $AD$ is negative, therefore the ray lies above $AD$. Further, the broken line $DB$ is obtained by shifting $AC$ by the vector $\overrightarrow{AD}$, therefore it is situated below the ray. Finally, the ray lies above the rest $B[1]$ of the graph. This implies that the ray lies above the graph of $u_k$.

Thus, we have verified that the function $u_{k}$ satisfies forward SIC. The proof by induction is complete.

Thus, $\tilde u_1$ satisfies forward SIC, and therefore each ray reflected from the part of $\graph(\tilde u_1)$ corresponding to $[-1,\, 0]$ goes above the point $(1,0)$. Since $\tilde u$ is convex and coincides with $\tilde u_1$ on $[-1,\, 0]$, we obtain that this ray also lies above $\graph(\tilde u)$, and $\tilde u$ satisfies forward SIC. From symmetry considerations we conclude that $\tilde u$ satisfies SIC.
\end{proof}

Consider a real value $x_0$ and a monotone non-decreasing function $u : [a,\, b] \to \RRR$ with $(x_0 + b)/2 \le a < b$,\, $u(b) \le 0$ and such that each particle reflected from $\graph(u)$ goes above the point $(x_0, 0)$. In analytical terms the latter condition means that
\begin{equation}\label{above_zero}
\frac{-2u(x)}{x - x_0} \le \frac{1 - u'^2(x)}{u'(x)}
\end{equation}
at each point $a < x < b$ where $u'(x)$ exists and is nonzero.

Notice that since $u$ is monotone, $u'$ is a measurable function defined almost everywhere.

Next we define the non-decreasing function $u_{(0)} : [a,\, b] \to \RRR$ such that $\graph(u_{(0)})$ coincides with an arc of the parabola through $(b, u(b))$ with focus at $(x_0,0)$ and with the vertical axis (see Fig.~\ref{fig:exlem30}). This means that $u_{(0)}$ has the form $u_{(0)}(x) = ((x - x_0)^2 - p^2)/(2p)$, where the positive value $p$ is determined from the condition $u_{(0)}(b) = u(b)$. In particular, we have
$$
\frac{-2u_{(0)}(x)}{x - x_0} = \frac{1 - u_{(0)}'^2(x)}{u_{(0)}'(x)};
$$
that is, {each} particle reflected from $\graph(u_{(0)})$ {\it passes through} the origin.

\begin{center}
\parbox{13cm}{\begin{picture}(0,140)
\rput(2.7,3.5){
\scalebox{0.8}{
 \psecurve[linewidth=0.3pt,linestyle=dashed,linecolor=blue]
 (2,-4.8)(3,-4.55)(4,-4.2)(5,-3.75)(6,-3.2)(7,-2.55)(8,-1.8)(9,-0.95)(10,0)(11,1.05)
 \psdots[dotsize=2.2pt](6,0)(9,0)(0,0)
   \rput(6,0.35){\scalebox{1.2}{$a$}}
\rput(-0.1,0.35){\scalebox{1.2}{$x_0$}}
   \rput(9,0.35){\scalebox{1.2}{$b$}}
     \rput(7.75,-3.1){\scalebox{1.2}{$y = u_{(0)}(x)$}}
     \rput(5,-2.8){\scalebox{1.2}{$y = u(x)$}}
\rput(11.5,0){\scalebox{1.2}{$x$}}
\psline[linewidth=1pt,linestyle=dotted](9,0)(9,-0.95)
\psline[linewidth=0.4pt,arrows=->,arrowscale=1.5](-0.5,0)(11,0)
\psline[linewidth=1pt,linestyle=dotted](6,0)(6,-3.2)

\psline[linewidth=0.5pt,linecolor=red,arrows=->,arrowscale=1.7](6.56,1.6)(6.56,-2.8)
\psline[linewidth=0.5pt,linecolor=red,arrows=->,arrowscale=1.7](6.56,-2.8)(0.05,-0.0375)
\psline[linewidth=0.5pt,linecolor=red,arrows=->,arrowscale=1.7](7,1.6)(7,-2.31)
\psline[linewidth=0.5pt,linecolor=red,arrows=->,arrowscale=1.7](7,-2.31)(3.5,1)

     \pscurve
     (6,-2.76)(7,-2.31)(8,-1.71)(9,-0.95)
     \psecurve
     (5,-3.75)(6,-3.2)(7,-2.55)(8,-1.8)(9,-0.95)(10,0)

      \scalebox{0.85}{
  \rput(-0.14,0){\psecurve[linewidth=0.353pt,linestyle=dashed,linecolor=blue]
 (2,-4.8)(3,-4.55)(4,-4.2)
 (5,-3.75)(6,-3.2)(7,-2.55)(8,-1.8)(9,-0.95)(10,0)(11,1.05)
 }}
     \scalebox{0.95}{
  \rput(-0.4,0){\psecurve[linewidth=0.316pt,linestyle=dashed,linecolor=blue]
 (2,-4.8)(3,-4.55)(4,-4.2)
 (5,-3.75)(6,-3.2)(7,-2.55)(8,-1.8)(9,-0.95)(10,0)(11,1.05)
 }}
     \scalebox{0.8}{
  \rput(-0.78,0){\psecurve[linewidth=0.375pt,linestyle=dashed,linecolor=blue]
 (2,-4.8)(3,-4.55)(4,-4.2)(5,-3.75)(6,-3.2)(7,-2.55)(8,-1.8)(9,-0.95)(10,0)(11,1.05)
 }}
     \scalebox{0.9}{
  \rput(-0.99,0){\psecurve[linewidth=0.33pt,linestyle=dashed,linecolor=blue]
 (2,-4.8)(3,-4.55)(4,-4.2)(5,-3.75)(6,-3.2)(7,-2.55)(8,-1.8)(9,-0.95)(10,0)(11,1.05)
 }}     }}
\end{picture}}
\vskip 0.4cm

\figure \label{fig:exlem30}
The functions $u$ and $u_0$ and two particles reflected from $\graph(u_0)$ and $\graph(u)$.
\end{center}

\begin{remark}\label{zam6}
In the particular case where $x_0 = -1$,\, $a \ge 0$,\, $u(b) = 0$,\, $b = 1$ the corresponding function $u_{(0)}(x)$ coincides with $u_0(x) = ((x+1)^2 - 4)/4$ on $[a,\, 1]$.
\end{remark}

The following lemma states in particular that the resistance of $u_{(0)}$ is smaller than the resistance of $u$.

Let $\varphi : [a,\, b] \to \RRR$ be a differentiable non-negative non-decreasing function.

\begin{lemma}\label{l3-0}
We have
$$
\int_a^{b} \frac{\varphi(x)}{1 + u_{(0)}'^2(x)}\, dx \le \int_a^{b} \frac{\varphi(x)}{1 + u'^2(x)}\, dx.
$$
\end{lemma}


\begin{proof}
Without loss of generality we assume that $x_0 = 0$; indeed, the general case is reduced to this one by the change of variables $x \mapsto x - x_0$.

It suffices to prove Lemma \ref{l3-0} for absolutely continuous functions $u$. Indeed, if $u$ is not absolutely continuous, substitute it with the function
$$
\tilde u(x) = u(b) - \int_x^b u'(\xi)\, d\xi.
$$
The resulting function $\tilde u$ is absolutely continuous. We have $\tilde u(b) = u(b)$, therefore the corresponding function $\tilde u_{(0)}$ coincides with $u_{(0)}$. Further, $\tilde u'(x) = u'(x)$ at each point where $u$ is differentiable; as a result we obtain
$$
\int_a^{b} \frac{\varphi(x)}{1 + \tilde u'^2(x)}\, dx = \int_a^{b} \frac{\varphi(x)}{1 + u'^2(x)}\, dx.
$$
Finally, we have $\tilde u(x) \ge u(x)$; therefore condition \eqref{above_zero} holds for $\tilde u$.

Suppose that the statement of Lemma \ref{l3-0} is true for absolutely continuous functions; then we have
$$
\int_a^{b} \frac{\varphi(x)}{1 + \tilde u_{(0)}'^2(x)}\, dx \le \int_a^{b} \frac{\varphi(x)}{1 + \tilde u'^2(x)}\, dx,
$$
and therefore this statement is also true for the functions $u$ and $u_{(0)}$.

Let now $u$ be absolutely continuous. Take the (unique) positive $p = p(x)$ such that
\begin{equation}\label{p}
u(x) = \frac{x^2 - p^2}{2p};
\end{equation}
it follows that
$$
\frac{-2u(x)}{x} = \frac{1 - (x/p)^2}{(x/p)},
$$
and taking into account \eqref{above_zero} (with $x_0 = 0$) one comes to the inequalities $u'(x) \le x/p$ and $x/p < 1$; hence
\begin{equation}\label{2}
0 \le u'(x) \le x/p < 1.
\end{equation}

Define the 1-parameter family of functions
\begin{equation}\label{1par}
u(x,t) = \left\{
\begin{array}{cc}
\frac{x^2 - p^2(t)}{2p(t)}, & \text{if } \, a \le x \le t\\
u(x), & \text{if } \, t \le x \le b
\end{array}
\right.,
\end{equation}
with $a \le t \le b$ (see Fig.~\ref{fig:par}). Using that $u(b,b) = u(b)$, one finds

$$
u(x,a) = u(x) \quad \text{and} \quad u(x,b) = u_{(0)}(x).
$$

\begin{center}
\centering
\parbox{13cm}
{\begin{picture}(0,145)
\rput(3,4){
\scalebox{0.8}{
 \psecurve[linewidth=0.3pt,linestyle=dashed,linecolor=blue]
 (-1,-4.95)(0,-5)(1,-4.95)(2,-4.8)(3,-4.55)(4,-4.2)(5,-3.75)(6,-3.2)(7,-2.55)(8,-1.8)(9,-0.95)(10,0)(11,1.05)
 \psdots[dotsize=2pt](10,0)(5,0)
\rput(10,0.35){\scalebox{1.2}{$b$}}
\rput(5,0.35){\scalebox{1.2}{$a$}}
\rput(-0.3,0.05){\scalebox{1.2}{0}}
\rput(6.9,0.35){\scalebox{1.2}{$t$}}
\rput(11.5,0){\scalebox{1.2}{$x$}}
\psline[linewidth=0.6pt,linestyle=dotted](6.95,0)(6.95,-1.81)
\psline[linewidth=0.4pt,arrows=->,arrowscale=1.5](0,-5.3)(0,1)
\psline[linewidth=0.4pt,arrows=->,arrowscale=1.5](0,0)(11,0)
\psline[linewidth=0.6pt,linestyle=dotted](5,0)(5,-3.75)
   \pscurve[linewidth=1pt,linecolor=red]
   (6.95,-1.81)(8,-1.29)(9,-0.75)(9.5,-0.47)
   \pscurve[linewidth=0.9pt,linestyle=dashed,linecolor=red]
   (5,-2.625)(6,-2.24)(6.95,-1.81)
\scalebox{0.5}{
  \rput(-0.15,0){\psecurve[linewidth=0.6pt,linestyle=dashed,linecolor=blue]
 (-1,-4.95)(0,-5)(1,-4.95)(2,-4.8)(3,-4.55)(4,-4.2)(5,-3.75)(6,-3.2)(7,-2.55)(8,-1.8)(9,-0.95)(10,0)(11,1.05)
 }}
 \scalebox{0.6}{
  \rput(-0.5,0){\psecurve[linewidth=0.5pt,linestyle=dashed,linecolor=blue]
 (-1,-4.95)(0,-5)(1,-4.95)(2,-4.8)(3,-4.55)(4,-4.2)(5,-3.75)(6,-3.2)(7,-2.55)(8,-1.8)(9,-0.95)(10,0)(11,1.05)
 }}
\scalebox{0.7}{
  \rput(-0.75,0){\psecurve[linewidth=0.43pt,linestyle=dashed,linecolor=blue]
 (-1,-4.95)(0,-5)(1,-4.95)(2,-4.8)(3,-4.55)(4,-4.2)(5,-3.75)(6,-3.2)(7,-2.55)(8,-1.8)(9,-0.95)(10,0)(11,1.05)
 }}
\scalebox{0.8}{
  \rput(-0.95,0){\psecurve[linewidth=0.375pt,linestyle=dashed,linecolor=blue]
 (-1,-4.95)(0,-5)(1,-4.95)(2,-4.8)(3,-4.55)(4,-4.2)(5,-3.75)(6,-3.2)(7,-2.55)(8,-1.8)(9,-0.95)(10,0)(11,1.05)
 }}
\scalebox{0.9}{
  \rput(-1.15,0){\psecurve[linewidth=0.375pt,linestyle=dashed,linecolor=blue]
 (-1,-4.95)(0,-5)(1,-4.95)(2,-4.8)(3,-4.55)(4,-4.2)(5,-3.75)(6,-3.2)(7,-2.55)(8,-1.8)(9,-0.95)(10,0)(11,1.05)
 \psecurve[linewidth=1pt,linecolor=red](5,-3.75)(5.55,-3.46)(6,-3.2)(7,-2.55)(7.7,-2.0355)(8,-1.8)
 }}
}}
\end{picture}}
\vskip 0.4cm
\figure \label{fig:par}
The family of functions $u(x,t)$.
\end{center}

Denote
$$
R(t) = \int_{a}^b \frac{\varphi(x)}{1 + u'_x(x,t)^2}\, dx, \quad a \le t \le b.
$$
We have
$$
R(a) = \int_{a}^b \frac{\varphi(x)}{1 + u'^2(x)}\, dx \quad \text{and} \quad R(b) = \int_{a}^b \frac{\varphi(x)}{1 + u_{(0)}'^2(x)}\, dx.
$$
To prove the lemma, it suffices to show that $R(a) \ge R(b)$.

One has
$$
u'_x(x,t) = \left\{
\begin{array}{cc}
x/p(t), & \text{if } \, a \le x \le t\\
u'(x), & \text{if } \, t \le x \le b
\end{array}
\right.;
$$
hence
\begin{equation}\label{eqR}
R(t) = \int_{a}^t \frac{\varphi(x)}{1 + x^2/p^2(t)}\, dx + \int_t^b \frac{\varphi(x)}{1 + u'^2(x)}\, dx.
\end{equation}
We have
$$
R'(t) = -\frac{\varphi(t)}{1 + u'^2(t)} + \frac{\varphi(t)}{1 + t^2/p^2(t)} + 2p(t)p'(t) \int_a^t \frac{x^2 \varphi(x)}{(x^2 + p^2(t))^2}.
$$
One has
\begin{equation}\label{eqP}
p(t) = -u(t) + \sqrt{t^2 + u^2(t)};
\end{equation}
hence
$$
p' = -u' + \frac{t + u'u}{\sqrt{t^2 + u^2}}.
$$

Since the function $u$ is absolutely continuous, so also is the function $p$ \eqref{eqP}, and therefore so is the first term in \eqref{eqR}. Further, the second term in \eqref{eqR} is absolutely continuous as a primitive of the bounded measurable function $1/(1 + u'^2(x))$. Therefore the function $R$ is absolutely continuous, and by the Newton-Leibniz formula we have
$$R(b) - R(a) = \int_a^b R'(t)\, dt.$$
Thus, it suffices to prove that $R'(t) \le 0$ for almost all $a < t < b$.

Introducing $u'(t) = w$ and $t/p(t) = \tau$, one obtains (in what follows we put $p = p(t)$ and $p' = p'(t)$)
$$
p' = \frac{t - wp}{p + u} = 2p\, \frac{t - wp}{t^2 + p^2} 
= 2\, \frac{\tau - w}{1 + \tau^2} \quad\quad (0 \le w \le \tau \le 1)
$$
and
\begin{equation}\label{3}
-\frac{R'(t)}{\varphi(t)} = \left( \frac{1}{1 + w^2} - \frac{1}{1 + \tau^2} \right) + 4\, \frac{w - \tau}{1 + \tau^2} \int_{a/p}^{\tau} \frac{\xi^2}{(1 + \xi^2)^2}\, \frac{\varphi(p\xi)}{\varphi(p\tau)}\, d\xi.
\end{equation}
Denote the expression in the right hand side of \eqref{3} by $\Phi(w) = \Phi(w,\tau,p).$ To complete the proof of the lemma, one needs to show that $\Phi(w) \ge 0$ for all $w \in [0,\, \tau]$.

Substituting $w = \tau$, one obtains $\Phi(w=\tau) = 0$. Substituting $w = 0$, one has
\begin{equation}\label{7}
\frac{1 + \tau^2}{2\tau}\, \Phi(w=0) = \frac{\tau}{2} - 2 \int_{a/p}^{\tau} \frac{\xi^2}{(1 + \xi^2)^2}\, \frac{\varphi(p\xi)}{\varphi(p\tau)}\, d\xi.
\end{equation}
One obviously has
$$
0 < \frac{a}{p} \le \tau \le 1.
$$
Substituting $\tau = a/p$, we see that the expression in the right hand side of \eqref{7} takes the positive value $\tau/2$. Further, the derivative of this expression by $\tau$ is equal to the non-negative value
$$
\frac{(1 - \tau^2)^2}{2(1 + \tau^2)^2} + \frac{2p\varphi'(p\tau)}{\varphi^2(p\tau)} \int_{a/p}^{\tau} \frac{\xi^2}{(1 + \xi^2)^2}\, \varphi(p\xi)\, d\xi.
$$
Thus, we conclude that $\Phi(w=0) > 0$.

\begin{center}
	\centering
\parbox{13cm}
{\begin{picture}(0,140)
\rput(6,0.5){
\scalebox{1.7}{
\psecurve[linewidth=0.4pt]
(-5,0.276)(-4,0.4)(-3,0.6154)(-2,1)(-1.5,1.28)(-1,1.6)(-0.5,1.8824)(0,2)(0.5,1.8824)(1,1.6)(1.5,1.28)(2,1)(3,0.6154)(4,0.4)(5,0.276)
\psline[linewidth=0.4pt,arrows=->,arrowscale=1.2](-4,0)(4.5,0)
\rput(4.5,-0.15){\scalebox{0.6}{$w$}}
\rput(2,-0.15){\scalebox{0.6}{1}}
\rput(0,-0.15){\scalebox{0.6}{0}}
\rput(1.5,-0.15){\scalebox{0.6}{$\tau$}}
\rput(-0.12,2.15){\scalebox{0.6}{1}}
\rput(-0.12,1.8){\scalebox{0.6}{$\bbb$}}
\psline[linewidth=0.4pt,arrows=->,arrowscale=1.2](0,0)(0,2.5)
\psline[linewidth=0.2pt,linestyle=dashed,dash=2pt](0,2)(2,0.99)
\psline[linewidth=0.2pt,linestyle=dashed,dash=2pt](1.5,1.28)(1.5,0)
\psdots[dotsize=1.5pt](0,2)(2,1)(2,0)
\psline[linewidth=0.4pt](0,1.8)(1.5,1.28)
\rput(-2.7,0.95){\scalebox{0.6}{$\Phi_1(w)$}}
\rput(0.55,1.35){\scalebox{0.6}{$\Phi_2(w)$}}
}}
\end{picture}}

\figure \label{fig:2f}
The functions $\Phi_1(w) = 1/(1 + w^2)$ and $\Phi_2(w) = kw + \bbb$ on the interval $[0,\, \tau]$.
\end{center}

Notice that $\Phi$ can be represented as the difference of two functions, $\Phi(w) = \Phi_1(w) - \Phi_2(w)$, where $\Phi_1(w) = 1/(1 + w^2)$ and $\Phi_2(w) = kw + \bbb$, with $k = k(\tau,p)$ and $\bbb = \bbb(\tau,p)$ (see Fig.~\ref{fig:2f}). Since $\Phi_1(\tau) = \Phi_2(\tau)$ and $\Phi_2(0) < \Phi_1(0) = 1$, one has
$$
\Phi_2'(\tau) = \frac{\Phi_2(\tau) - \Phi_2(0)}{\tau} > \frac{\Phi_1(\tau) - 1}{\tau} = -\frac{\tau}{1 + \tau^2}.
$$
On the other hand, using that $0 < \tau < 1$, one obtains
$$
\Phi_1'(\tau) = -\frac{2\tau}{(1 + \tau^2)^2} < -\frac{\tau}{1 + \tau^2} < \Phi_2'(\tau).
$$
It follows that $\Phi_1(w) > \Phi_2(w)$ in a left neighborhood of $\tau$.

Assume that $\Phi_1(w) < \Phi_2(w)$ at a point $w \in (0,\, \tau)$; then there exist at least two points $w' \in (0,\, w)$ and $w'' \in (w,\, \tau)$ where $\Phi_1 = \Phi_2$. Thus, first, $\Phi_2' = k < 0$ and second, there exist three points $w_1 \in (0,\, w')$,\, $w_2 \in (w',\, w'')$, and $w_3 \in (w'',\, \tau)$ where the equation $\Phi_1' = \Phi_2'$ is true. However, the function $\Phi_1'$ first decreases and then increases on $(0,\, +\infty)$, and therefore takes a fixed value in at most two points. This contradiction proves that $\Phi_1(w) \ge \Phi_2(w)$, and so, $\Phi(w) \ge 0$ for $w \in (0,\, \tau)$. Lemma \ref{l3-0} is proved.
\end{proof}

For any even admissible function $u$ and $\ve > 0$ one can, using Corollary \ref{cor}, find a piecewise linear even admissible function $\hat u$ such that
$\FFF(\hat u) < \FFF(u) + \ve$. Further we use Lemma \ref{l2} to find a convex even admissible function $\tilde u$ continuous at $x=1$ and such that $\FFF(\tilde u) \le \FFF(\hat u)$. Next applying Lemma \ref{l3-0} to the restriction of $\tilde u$ on $[0,\, 1]$ with $\varphi = f$ and $x_0 = -1$, and taking account of Remark \ref{zam6} we obtain
$$
\int_0^1 \frac{f(x)}{1 + u_0'^2(x)}\, dx \le \int_0^1 \frac{f(x)}{1 + \tilde u'^2(x)}\, dx,
$$
and using that $f,\, \tilde u$ and $u_0$ are even, we conclude that $\FFF(u_0) \le \FFF(\tilde u)$. Thus, $\FFF(u_0) < \FFF(u) + \ve$. Taking $\ve$ arbitrary small, one has $\FFF(u_0) \le \FFF(u)$. Theorem \ref{t1} is proved.

\section*{Appendix}

Let $\Om$ be a bounded domain in $\RRR^d$,\, $d \ge 3$. First consider the case where $\Om$ is a cartesian product $\Om = \Om_1 \times \Om_2$, where $\Om_1 \subset \RRR^2$ and $\Om_2 \subset \RRR^{d-2}$. For a two-dimensional domain $\Om_1$ it is known \cite{PkakSIAM} that $\inf_u F(u; \Om_1) = 1/2$; therefore for any $\ve > 0$ there is an admissible function $u_1$ on $\bar\Om_1$ such that $F(u_1; \Om_1) < 1/2 + \ve$.

Define the function $u$ on $\bar\Om$ by $u(x) = u_1(x_1)$, if $x = (x_1, x_2) \in \Om$\, $(x_1 \in \RRR^2,\, x_2 \in \RRR^{d-2})$ and $u(x) = 0$, if $x \in \pl\Om$. It is easy to check that $u$ satisfies SIC, and therefore is admissible. One has $|\nabla u(x)| = |\nabla u_1(x_1)|$, and so,
\begin{multline*}
F(u; \Om) = \frac{1}{|\Om_1 \times \Om_2|} \int_{\Om_1 \times \Om_2} \frac{dx}{1 + |\nabla u(x)|^2}
= \frac{1}{|\Om_1|} \int_{\Om_1} \frac{dx_1}{1 + |\nabla u_1(x_1)|^2} < 1/2 + \ve.
\end{multline*}
Since one always has $F(u; \Om) > 1/2$, we conclude that $\inf_u F(u; \Om) = 1/2$.

Let now $\Om$ be an arbitrary domain in $\RRR^d$. Take a cubic lattice $(\del\ZZZ)^d$ and denote by $\Om_1,\ldots,\Om_n$ the cubes of the lattice that are contained in $\Om$. Choose $\del > 0$ so small that the volume $|\Om \setminus (\cup_{i=1}^n \Om_i)|$ is smaller than $\frac{\ve}{2}\, |\Om|$.

Since each cube $\Om_i$ is the cartesian product of a two-dimensional square and a $(d-2)$-dimensional cube, we have $\inf_u F(u; \Om_i) = 1/2$ and one can choose an admissible function $u_i$ on $\bar\Om_i$ such that $F(u_i; \Om_i) < 1/2 + \ve/2$. Define the function $u$ on $\bar\Om$ as follows: the restriction $u\rfloor_{\Om_i}$ coincides with $u_i$ for each $i$;\, $u\rfloor_{\pl\Om} = 0$; the restriction $u\rfloor_{\Om\setminus(\cup_{i}\Om_i)}$ coincides with a negative constant. The so-defined function $u$ satisfies SIC, and therefore is admissible. We have
$$
F(u; \Om) = \frac{1}{|\Om|} \left( \sum_{i=1}^n \int_{\Om_i} \frac{dx}{1 + |\nabla u_i(x)|^2} + \int_{\Om\setminus(\cup_{i}\Om_i)} dx \right)
$$
$$
= \sum_{i=1}^n \frac{|\Om_i|}{|\Om|}\, F(u_i; \Om_i) + \frac{1}{|\Om|} \cdot |\Om\setminus(\cup_{i=1}^n \Om_i)|
$$
$$
< (1/2 + \ve/2) + \ve/2 = 1/2 + \ve.
$$
Due to arbitrariness of $\ve > 0$, we conclude that $\inf_u F(u; \Om) = 1/2$.

\section*{Acknowledgements}
The work of AP was supported by Portuguese funds through CIDMA -- Center for Research and Development in Mathematics and Applications and FCT -- Portuguese Foundation for Science and Technology, within the project PEst-OE/MAT/UI4106/2014, as well as by the FCT research project PTDC/MAT/113470/2009.

\end{document}